%% file: FriSchShi23_preprint3.tex
\documentclass[11pt]{article}

\usepackage{float}
\usepackage[utf8]{inputenc}
\usepackage[T1]{fontenc}
\usepackage{amsmath,amssymb}
\usepackage{graphicx}
\usepackage{float}
\numberwithin{equation}{section}
\usepackage[hyperref,amsmath,thmmarks]{ntheorem}
\usepackage{aliascnt}
\usepackage[a4paper,centering,bindingoffset=0cm,marginpar=2cm,margin=2.5cm]{geometry}
\usepackage[pagestyles]{titlesec}
\usepackage[font=footnotesize,format=plain,labelfont=sc,textfont=sl,width=0.75\textwidth,labelsep=period]{caption}

\usepackage{subfig}

\usepackage[backend=biber,maxnames=10,backref=true,hyperref=true,giveninits=true,safeinputenc]{biblatex}
\DefineBibliographyStrings{english}{%
	backrefpage = {cited on page},
	backrefpages = {cited on pages},
}

\title{Quadratic neural networks for solving inverse problems}
\author{Leon Frischauf$^1$\\{\footnotesize\href{mailto:leon.frischauf@univie.ac.at}{leon.frischauf@univie.ac.at}}
	\and Otmar Scherzer$^{1,2,3}$\\{\footnotesize\href{mailto:otmar.scherzer@univie.ac.at}{otmar.scherzer@univie.ac.at}}
	\and Cong Shi$^{1,2,*}$\\{\footnotesize\href{mailto:cong.shi@univie.ac.at}{cong.shi@univie.ac.at}}}
\date{}

\DeclareFieldFormat[report]{title}{``#1''}
\DeclareFieldFormat[book]{title}{``#1''}
\AtEveryBibitem{\clearfield{url}}
\AtEveryBibitem{\clearfield{note}}

\titleformat{\section}[block]{\large\sc\filcenter}{\thesection.}{0.5ex}{}[]
\titleformat{\subsection}[runin]{\bf}{\thesubsection.}{0.5ex}{}[.]

\usepackage[pdftex,colorlinks=true,linkcolor=blue,citecolor=green,urlcolor=blue,bookmarks=true,bookmarksnumbered=true]{hyperref}
\hypersetup
{
	pdfauthor={Authors},
	pdfsubject={Subject},
	pdftitle={Title},
	pdfkeywords={Keywords}
}

\newpagestyle{headers}
{
	\headrule
	\sethead[\footnotesize\thepage][\footnotesize\sc Authors][]{}{\footnotesize\sc Quadratic neural networks}{\footnotesize\thepage}
	\setfoot{}{}{}
}
\include{header.tex}

\begin{document}

	\maketitle
	\thispagestyle{empty}
	\begin{center}{
			\hspace*{2em}
			\parbox[t]{10em}{\footnotesize
				\hspace*{-1ex}$^1$Faculty of Mathematics\\
				University of Vienna\\
				Oskar-Morgenstern-Platz 1\\
				A-1090 Vienna, Austria\\\vspace*{1ex}}
			\hfil
			\parbox[t]{12em}{\footnotesize
				\hspace*{-1ex}$^2$Johann Radon Institute for\\
				\hspace*{0.5em}Computational and Applied\\
				\hspace*{0.5em}Mathematics (RICAM)\\
				Altenbergerstraße 69\\
				A-4040 Linz, Austria}
			\hfil
			\parbox[t]{12em}{\footnotesize
				\hspace*{-1ex}$^3$Christian Doppler Laboratory for Mathematical Modeling and\\
				Simulation of Next Generations of Ultrasound Devices (MaMSi)\\
				Oskar-Morgenstern-Platz 1\\
				A-1090 Vienna, Austria\\\vspace*{1ex}}
	}\end{center}
	\blfootnote{${}^*$Cong Shi is the corresponding author.}
	
	\begin{abstract}
	In this paper we investigate the solution of inverse problems with neural network ansatz functions with
	generalized decision functions. The relevant observation for this work is that such functions can approximate typical test cases, such as the Shepp-Logan
		phantom, better, than standard neural networks.
    Moreover, we show that the convergence analysis of numerical methods for solving inverse problems with shallow generalized neural network functions leads to more intuitive convergence conditions, than for deep affine linear neural networks.
	\end{abstract}
\hspace*{5.8ex}\textbf{Keywords:} Inverse problems; generalized neural network functions.\\
\hspace*{5.8ex}\textbf{MSC:} 65J22; 45Q05; 42C40; 65F20
	
\section{Introduction}
We consider an \emph{inverse problem}, which consists in solving an operator equation
\begin{equation}\label{eq:op}
	\A (\bx) = \by_0,
\end{equation}
where $\A: {\bf X}\rightarrow {\bf Y}$ denotes an operator, mapping between function spaces $({\bf X}, \norm{\cdot})$
and $({\bf Y}, \norm{\cdot})$.
For the numerical solution of \autoref{eq:op} an appropriate set of ansatz-functions $\bP$ has to be selected, on which an approximate solution of \autoref{eq:op} is calculated, for instance a function ${\bf p} \in \bP \subseteq {\bf X}$, which solves
\begin{equation}\label{eq:opapp}
	\A ({\bf p}) = \by,
\end{equation}
where $\by$ is an approximation of $\by_0$.
The history on the topic of approximating the solution of \autoref{eq:op} by discretization and regularization is long: For linear inverse problems this has been investigated for instance in \cite{Nat77}. Much later on neural network ansatz functions have been used for solving \autoref{eq:op}; see for instance \cite{ObmSchwHal21}. In this paper we investigate solving \autoref{eq:opapp} when $\bP$ is a set of neural network functions with \emph{generalized} decision functions. In order to spot the difference to classical neural network functions we recall the definition of them first:
\begin{definition}[Affine linear neural network functions] \label{de:affinenns}
Let $m \geq n \in \N$.
	Line vectors in $\R^m$ and $\R^n$ are denoted by
	\begin{equation*}
		\bw = (w^{(1)},w^{(2)},\ldots,w^{(m)})^T \text{ and }
		\vx = (x_1,x_2,\ldots,x_n)^T, \text{ respectively.}
	\end{equation*}
\begin{itemize}
	\item A \emph{shallow} affine linear neural networks ({\bf ALNN}) is a function (here $m=n$)
	 \begin{equation}\label{eq:classical_approximation}
		\vx \in \R^n \to {\bf p}(\vx) := \Psi[\vp](\vx) := \sum_{j=1}^{N} \alpha_j\sigma\left(\bw_j^T \vx +\theta_j \right),
	 \end{equation}
     with $\alpha_j, \theta_j \in \R$, $\bw_j \in \R^n$; $\sigma$ is an activation function, such as the sigmoid, ReLU${}^k$ or Softmax function.
     Moreover,
     \begin{equation} \label{eq:p}
     	\vp = [\alpha_1,\ldots,\alpha_N; \bw_1,\ldots,\bw_N; \theta_1,\ldots,\theta_N] \in \R^{n_*} \text{ with }
     	n_* = (n+2)N
     \end{equation}
     denotes the according parametrization of ${\bf p}$.
     In this context the set of {\bf ALNN}s is given by
     \begin{equation*}
     	\bP := \set{{\bf p} \text{ of the form \autoref{eq:classical_approximation}}: \vp \in \R^{n_*}}.
     \end{equation*}
     \item More recently \emph{deep} affine linear neural network functions ({\bf DNN}s) have become popular: An $(L+2)$-layer network looks as follows:
     \begin{equation}\label{eq:DNN}
     	\begin{aligned}
          \vx \in \R^n \to & {\bf p}(\vx) := \\
           &\Psi[\vp](\vx) := \sum_{j_L=1}^{N_L} \alpha_{j_L,L} \sigma_L\left( p_{j_L,L} \left( \sum_{j_{L-1}=1}^{N_{L-1}} \ldots \left( \sum_{j_1=1}^{N} \alpha_{j_1,1} \sigma_1 \left(p_{j_1,1}(\vx) \right) \right)\right) \right).
        \end{aligned}
     \end{equation}	
where
$$p_{j,l}(\vx) = \bw_{j,l}^T \vx +\theta_{j,l}
\text{ with } \alpha_{j,l}, \theta_{j,l} \in \R \text{ and } \bw_{j,l} \in \R^n \text{ for all } l=1,\ldots,L.$$
Here $s\in \R \to \sigma_k(s)$, $k=1,\ldots,L$, denote activation functions. Moreover, we denote the parametrizing vector by
    \begin{equation} \label{eq:pd}
	\vp = [\alpha_{1,1},\ldots,\alpha_{N_L,L}; \bw_{1,1},\ldots,\bw_{N_L,L}; \theta_{1,1},\ldots,\theta_{N_L,L}] \in \R^{n_*}.
\end{equation}

A shallow linear neural network can also be called a \emph{3-layer network}. The notation 3-layer network is consistent with the literature because input-and output are counted on top. Therefore,
a general $(L+2)$-layer network has only $L$ \emph{internal} layers.
\end{itemize}
\end{definition}
We demonstrate exemplary below (see \autoref{sec:mot}) that the convergence analysis of an iterative Gauss-Newton method for
solving a linear inverse problem restricted to a set of deep neural network ansatz functions is involved (see \autoref{ex:ccdn}), while the analysis with shallow neural networks ansatz functions leads to intuitive conditions (see also \cite{SchHofNas23}). Moreover, it is well-known in regularization theory that the quality of approximation of the solution of \autoref{eq:opapp} is depending on the approximation qualities of the set $\bP$:
Approximation of functions with shallow affine linear neural networks is a classical topic of machine learning and
in approximation theory, see for instance \cite{Pin99,Bar93,Cyb89,HorStiWhi89b,LesLinPinScho93,Mha93,SieXu22,SieXu23}. The central result  concerns the \emph{universal approximation property}. Later on the universal approximation property has been established for different classes of neural networks: Examples are dropout neural networks (see \cite{WagWanLia13,ManPelPorSanSen22}), convolutional neural networks (CNN) (see for example \cite{Zho18, Zho20}), recurrent neural networks (RNN) (see \cite{SchaZim07, Ham00}), networks with random nodes (see \cite{Whi89}), with random weights and biases (see \cite{PaoParSob94,IgePao95}) and with fixed neural network topology (see \cite{GelMaoLi99}), to name but a few.

In this paper we follow up on this topic and analyze neural network functions with decision functions, which are higher order polynomials, as specified now:
\begin{definition}[Shallow generalized neural network function] \label{de:gnn} Let $n,m,N \in \N$. Moreover, let
	$$\vfm_j: \R^n \to \R^m, \quad j=1,\ldots,N$$ be vector-valued functions. Neural network functions associated to $\set{\vfm_j: j=1,\ldots,N}$ are defined by
	\begin{equation}\label{eq:general}
		\vx \to {\bf p}(\vx) = \Psi[\vp](\vx) := \sum_{j=1}^{N} \alpha_j\sigma\left(\bw_j^T\vfm_j(\vx) +\theta_j\right) \text{ with } \alpha_j, \theta_j \in \R \text{ and } \bw_j \in \R^{m}.
	\end{equation}
    Again we denote the parametrization vector, containing all coefficients, with $\vp$ and the set of all such functions with $\bP$.
    We call
    	\begin{equation}\label{eq:GeneralFunctions}
    		\mathcal{D}: = \set{ \vx \to \bw^T \vfm_j(\vx) +\theta: \bw \in \R^{m}, \theta \in \R, j=1,\ldots,N}
    	\end{equation}
    	the set of \emph{decision functions} associated to \autoref{eq:general}.
    	The composition of $\sigma$ with a decision function is called \emph{neuron}.
\end{definition}
We discuss approximation properties of generalized neural networks (see \autoref{th:general}),
in particular neural network functions with decision functions, which are at most \emph{quadratic polynomials} in $\vx$. This idea is not completely new: In \cite{TsaTefNikPit19} neural networks with \emph{parabolic} decision functions have been applied in a number of applications. In \cite{FanXioWan20}, the authors proposed neural networks with radial decision functions, and proved an universal approximation result for ReLU-activated \emph{deep quadratic networks}.

\subsection*{Outline of this paper}
The motivation for this paper is to solve \autoref{eq:opapp} on a set of neural network functions $\bP$
with iterative algorithms, such as for instance a Gauss-Newton method. Good approximation
properties of functions of interest (such as the Shepp-Logan phantom) is provided with quadratic
\emph{decision functions}, which have been suggested already in \cite{Mha93}.
We concentrate on shallow neural networks because the resulting analysis of iterative methods for solving inverse problems is intuitive (see \autoref{sec:mot}). We apply different approximation theorems (such as for instance the universal approximation theorem of \cite{Cyb89}) to cover all situations from \autoref{de:quadratic}, below. Moreover, by constructing wavelet frames based on quadratic decision functions, we give an explicit convergence rate for shallow radial network approximations; see \autoref{sec:rates}, with main result \autoref{thm:ConvR}. In \autoref{sec:appendix} the important conditions of the \emph{approximation to the identity} (AtI) from \cite{DenHan09} used for proving the convergence rates in \autoref{sec:rates} are verified.

   \section{Examples of networks with generalized decision functions} \label{sex:examples}
    The versatility  of the network defined in \autoref{eq:general} is due to the flexibility of the functions $\vfm_j$, $j=1,\ldots,N$ to chose.

    We give a few examples.
    \begin{definition}[Neural networks with generalized decision functions] \label{de:quadratic}
    	We split them into three categories:
    	\begin{description}
    		\item{General quadratic neural networks {\bf (GQNN)}}: Let
    $$m = n+1, \quad \vfm_j(\vx) = \begin{pmatrix} f_j^{(1)}(\vx)\\ \vdots\\ f_j^{(n)}(\vx)\\ f_j^{(n+1)}(\vx) \end{pmatrix}= \begin{pmatrix} x_1\\ \vdots\\ x_n\\ \vx^T A_j \vx \end{pmatrix}
    \text{ for all } j=1,\ldots,N.$$
    That is, $\vfm_j$ is a graph of a quadratic function.
    Then \autoref{eq:general} reads as follows:
    \begin{equation}\label{eq:quadratic_approximation}
    	\begin{aligned}
    		\vx \to {\bf p}(\vx) = \Psi[\vp](\vx) &:=
    		\sum_{j=1}^{N} \alpha_j \sigma \left( \bw_j^T \vx + \vx^T A_j \vx +\theta_j\right) \\
    		&\text{ with }  \alpha_j, \theta_j \in \R, \bw_j \in \R^n, A_j \in \R^{n \times n} .
    	\end{aligned}
    \end{equation}
    $\vp$ denotes the parametrization of ${\bf p}$:
    \begin{equation} \label{eq:pq}
    	\begin{aligned}
    		\vp = [\alpha_1\ldots,\alpha_N; \bw_1,\ldots,\bw_N; A_1,\ldots,A_N;\theta_1,\ldots,\theta_N] \in \R^{n_*}
    		\text{ with }
    		n_* = (n^2+n+2)N.
    	\end{aligned}
    \end{equation}
    Note that in \autoref{eq:quadratic_approximation} all parameters of the entries of the matrices $A_j$ are parameters.
    \item{Constrained quadratic neural networks {\bf (CQNN)}:} These are networks, where the entries of the matrices $A_j$, $j=1,\ldots,N$ are constrained:
    	\begin{enumerate}
    		\item Let
    		$$\vfm(\vx)=\vfm_j(\vx) = \begin{pmatrix} f_j^{(1)}(\vx)\\ \vdots\\ f_j^{(n)}(\vx) \\ f_j^{(n+1)}(\vx)\end{pmatrix}=
    		\begin{pmatrix} x_1\\ \vdots\\ x_n \\ 0\end{pmatrix} \text{ for all } j=1,\ldots,N.$$
    		That is $A_j =0$ for all $j=1,\ldots,N$. This set of {\bf CQNN}s corresponds with the {\bf ALNN}s defined in \autoref{eq:classical_approximation}.
    		\item Let $A_j \in \R^{n \times n}$, $j=1,\ldots,N$ be chosen and fixed. We denote by {\bf MCNN} the family of functions
    		    \begin{equation}\label{eq:quadratic_approximation_fixed}
    			\begin{aligned}
    				\vx \to {\bf p}(\vx) = \Psi[\vp](\vx) &:=
    				\sum_{j=1}^{N} \alpha_j \sigma \left( \bw_j^T \vx + \xi_j \vx^T A_j \vx +\theta_j\right) \\
    				&\text{ with }  \alpha_j, \theta_j, \xi_j \in \R, \bw_j \in \R^n, A_j \in \R^{n \times n},
    			\end{aligned}
    		    \end{equation}
    		    which is parametrized by the vector
    		\begin{equation} \label{eq:pqc}
    			\begin{aligned}
    				\vp = [\alpha_1\ldots,\alpha_N; \bw_1,\ldots,\bw_N; \xi_1,\ldots,\xi_N; \theta_1,\ldots,\theta_N] \in \R^{n_*}
    				\text{ with }
    				n_* = (n+3)N.
    			\end{aligned}
    		\end{equation}
            In particular, choosing $A_j=\mathcal{I}$ we get:
            \begin{equation}\label{eq:radial_approximation}
   	             \begin{aligned}
   		           \vx \to {\bf p}(\vx) = \Psi[\vp](\vx) &:=
   		            \sum_{j=1}^{N} \alpha_j \sigma \left(\bw_j^T \vx + \xi_j \norm{\vx}^2 +\theta_j\right) \\
   	             	& \text{ with } \alpha_j, \theta_j \in \R, \bw_j \in \R^n, \xi_j \in \R.
   	             \end{aligned}
           \end{equation}
       \end{enumerate}
   \item{Radial quadratic neural networks {\bf (RQNN)s}:}
   For $\xi_j \neq 0$ the argument in \autoref{eq:radial_approximation} rewrites to
    			\begin{equation}\label{eq:nus}
    				\begin{aligned}
    					\nu_j(\vx) := \xi_j \norm{\vx}^2 + \hat{\bw}_j^T\vx  + \theta_j
    					= \xi_j \norm{\vx - \vy}^2 + \kappa_j
    				\end{aligned}
    			\end{equation}
    			with
    			\begin{equation} \label{eq:kappa}
    				\vy_j = - \frac{1}{2\xi_j} \hat{\bw}_j^T \text{ and } \kappa_j =
    				\theta_j - \frac{\norm{\hat{\bw}_j}^2}{4 \xi_j}.
    			\end{equation}
    		    We call the set of functions from \autoref{eq:radial_approximation}, which satisfy
    		    \begin{equation} \label{eq:circle}
    		    		\xi_j \geq 0 \text{ and } \kappa_j \leq 0 \text{ for all } j=1,\ldots,N,
    		    \end{equation}
                {\bf RQNN}s, radial neural network functions, because the level sets of $\nu_j$ are circles. These are \emph{radial basis functions}
                (see for instance \cite{Buh03} for a general overview).
   \item{Sign based quadratic {\bf (SBQNN)} and cubic neural networks {\bf (CUNN)}}: Let $m=n$.
            \begin{enumerate}
    		\item {\bf (SBQNN)}: Let
    		\begin{equation} \label{eq:sgn_distance}
    			f^{(i)} (\vx)= \text{sgn}(x_i) x_i^2 \text{ for } i=1,\ldots,n,
    		\end{equation}
    		or alternatively
    		\begin{equation*}
    			f^{(i)}(\vx) = \norm{\vx}^2 \text{sgn}(x_i) \text{ for } i=1,\ldots,n.
    		\end{equation*}
    		In the first case, and in the second case similarly, we obtain the family of functions
    		\begin{equation}\label{eq:generalf}
    			\Psi[\vp](\vx) := \sum_{j=1}^{N} \alpha_j \sigma
    			\left(\sum_{i=1}^n w_j^{(i)}\text{sgn} (x_i) x_i^2+\theta_j\right) \text{ with } \alpha_j, \theta_j \in \R \text{ and } \bw_j \in \R^n.
    		\end{equation}
    		We call these functions \emph{signed squared neural networks}. Note, here $\vp \in \R^{n_*}$ with $n_*=(n+2)N$.
    		   		\item {\bf (CUNN)}: Let
    		\begin{equation} \label{eq:cube}
    			f^{(i)} (\vx)= x_i^3 \text{ for } i=1,\ldots,n.
    		\end{equation}
    	    We obtain the family of functions
    		\begin{equation}\label{eq:cubef}
    			\Psi[\vp](\vx) := \sum_{j=1}^{N} \alpha_j \sigma
    			\left(\sum_{i=1}^n w_j^{(i)} x_i^3+\theta_j\right) \text{ with } \alpha_j, \theta_j \in \R \text{ and } \bw_j \in \R^n.
    		\end{equation}
    		We call these functions \emph{cubic neural networks}. Again $\vp \in \R^{n_*}$ with $n_*=(n+2)N$.
    	\end{enumerate}	
    \end{description}
    \end{definition}
    \begin{remark} \label{re:versatile}
    	\begin{enumerate}
    		\item It is obvious that generalized quadratic neural network functions and matrix constrained neural networks
    		are more versatile than affine linear works, and thus they satisfy the universal approximation property (see \cite{Cyb89}).
    		\item  The constraint \autoref{eq:circle} does not allow to use the universal approximation theorem in a straight forward manner. In fact we prove an approximation result indirectly via a convergence rates result (see \autoref{sec:rates}).
    		\item Sign based quadratic and cubic neural networks satisfy the universal approximation property, which is proven below in \autoref{th:qdf} by reducing it to the classical result from \cite{Cyb89}.
    	\end{enumerate}
    \end{remark}

In order to prove the universal approximation property of {\bf SBQNN}s and {\bf CUNN}s we review the \emph{universal approximation theorem} as formulated by \cite{Cyb89} first. It is based discriminatory properties of the
functions $\sigma$.
\begin{definition}[Discriminatory function] \label{de:disc_function}
	A function $\sigma : \R \to \R$ is called \emph{discriminatory} (see \cite{Cyb89}) if every measure $\mu$ on $[0,1]^n$, which satisfies
	\begin{equation*}
		\int_{[0,1]^n} \sigma (\bw^T \vx + \theta)\,d\mu(\vx)=0 \quad \text{ for all } \bw \in \R^n \text{ and } \theta \in \R
	\end{equation*}
	implies that $\mu \equiv 0$.
\end{definition}
\begin{example} \label{ex:sigmoid}
	Note that every non-polynomial function is \emph{discriminatory} (this follows from the results in \cite{LesLinPinScho93}).
	Therefore the choices of activation function in \autoref{de:affinenns}
	are discriminatory for the Lebesgue-measure.
\end{example}
With these basic concepts we are able to recall Cybenko's universal approximation result.

\begin{theorem}[\cite{Cyb89}]\label{le:LinearApproximation}
	Let $\sigma:\R \to \R$ be a continuous discriminatory function. Then, for
	every function $g \in C([0,1]^n)$ and every $\epsilon>0$, there exists a function
	\begin{equation}\label{eq:Linear}
		G_\epsilon(\vx) = \sum_{j=1}^{N}\alpha_j \sigma(\bw_j^T\vx + \theta_j) \qquad \text{ with } N \in \N, \alpha_j, \theta_j \in \R, \bw_j \in \R^n,
	\end{equation}
	satisfying
	\begin{equation*}
		|G_\epsilon(\vx)-g(\vx)|<\epsilon \text{ for all } \vx\in [0,1]^n.
	\end{equation*}
\end{theorem}
In the following we formulate and prove a modification of Cybenko's universal approximation result \cite{Cyb89} for shallow \emph{generalized} neural networks as introduced in \autoref{de:gnn}.

\begin{theorem}[Generalized universal approximation theorem]\label{th:general}
	Let $\sigma:\R \to \R$ be a continuous discriminatory function and assume that $\vfm_j : [0,1]^n \to \R^{m}$, $j=1,\ldots,N$ are injective (this in particular means that $n \leq m$) and continuous, respectively. We denote $\vfm = (\vfm_1,\ldots,\vfm_N)$.
	
	Then for every $g\in C([0,1]^n)$ and every $\epsilon>0$ there exists some function
	\begin{equation*}
		\Psi^\vfm(\vx) := \sum_{j=1}^{N} \alpha_j\sigma\left(\bw_j^T\vfm_j(\vx) +\theta_j\right) \text{ with } \alpha_j, \theta_j \in \R \text{ and } \bw_j \in \R^{m}
	\end{equation*}
	satisfying
	\begin{equation*}
		\abs{\Psi^\vfm(\vx) - g(\vx)} < \epsilon \text{ for all }\vx\in [0,1]^n.
	\end{equation*}
\end{theorem}

\begin{proof}
	We begin the proof by noting that since $\vx \to \vfm_j(\vx)$ is injective, the inverse function on the range of $\vfm_j$ is well-defined, and we write $\vfm_j^{-1}:\vfm_j([0,1]^n) \subseteq \R^m \to [0,1]^n \subseteq \R^n$.
	The proof that $\vfm_j^{-1}$ is continuous relies on the fact that the domain $[0,1]^n$ of $\vfm_j$ is compact, see for instance \cite[Chapter XI, Theorem 2.1]{Dug78}. Then applying the Tietze–Urysohn–Brouwer extension theorem (see \cite{Kel55}) to the continuous function $g \circ \vfm_j^{-1} : \vfm_j([0,1]^n) \to \R$, we see that this function can be extended continuously to $\R^m$. This extension will be denoted by $g^*:\R^m \to \R$.
	
	We apply \autoref{le:LinearApproximation} to conclude that there exist $\alpha_j, \theta_j \in \R$ and $ \bw_j \in \R^{m}$, $j=1,\ldots,N$ such that
	\begin{equation*}
		\Psi^*(\vzm) := \sum_{j=1}^{N} \alpha_j \sigma(\bw_j^T \vzm +\theta_j) \text{ for all } \vzm \in \R^{m},
	\end{equation*}
	which satisfies
	\begin{equation}\label{eq:1}
		\abs{\Psi^*(\vzm)-g^*(\vzm)} <\epsilon \text{ for all } \vzm \in [0,1]^m.
	\end{equation}
	Then, because $\vfm_j$ maps into $\R^m$ we conclude, in particular, that
	\begin{equation*}
		\begin{aligned}
			\Psi^*(\vfm_j(\vx)) &=\sum_{j=1}^{N} \alpha_j \sigma(\bw_j^T \vfm_j(\vx)+\theta_j)
		\end{aligned}
	\end{equation*}		
	and
	\begin{equation*}
		\begin{aligned}
			\abs{\Psi^*(\vfm_j(\vx)) - g(\vx)}
			=\abs{\Psi^*(\vfm_j(\vx)) - g^*(\vfm_j(\vx))} <\epsilon.
		\end{aligned}
	\end{equation*}
	Therefore $\Psi^\vfm(\cdot):= \Psi^*(\vfm_j(\cdot))$ satisfy the claimed assertions.
\end{proof}
It is obvious that the full variability in $\bw$ is the key to bring our proof and the universal approximation theorem in context. That is, if $\bw_j$, $\theta_j$, $j=1,\ldots,N$ are allowed to vary over $\R^n$, $\R$, respectively. {\bf RQNN}s are constrained to $\theta_j < \norm{\vw_j}^2/(4 \xi_j)$ in \autoref{eq:kappa} and thus
	 \autoref{th:general} does not apply. Interestingly \autoref{thm:ConvR} applies and allows to approximate functions in $\mathcal{L}^1(\R^n)$ (even with rates).

\begin{corollary}[Universal approximation properties of {\bf SBQNN}s and {\bf CUNN}s] \label{th:qdf}
	Let the discriminatory function $\sigma:\R \to \R$ be Lipschitz continuous.
	All families of neural network functions from \autoref{de:quadratic} satisfy the universal approximation property on $[0,1]^n$.
\end{corollary}
\begin{proof} The proof follows from \autoref{th:general} and noting that all our functions $\bf{f}_j$, $j=1,\ldots,N$ defined in \autoref{de:quadratic} are injective.
\end{proof}

    \section{Motivation}

    \subsection{Motivation 1: The Shepp-Logan phantom} \label{sec:slp}
    
    Almost all tomographic algorithms are tested with the Shepp-Logan phantom (see \cite{SheLog74}).
    It is obvious that it can be exactly represented with a {\bf GQNN} with 10 coefficients, while we conjecture that it can neither be exactly represented with an {\bf ALNN} or a {\bf DNN} with a finite number of coefficients. This observation extends immediately to all function that contains ``localized'' features.
    In this sense sparse approximations with linear neural networks would be more costly in practice.

	\subsection{Motivation 2: The Gauss-Newton iteration} \label{sec:mot}
	
	Let us assume that $\A:{\bf X} \to {\bf Y}$ is a linear operator. We consider the solution $\bf p$ of \autoref{eq:opapp} to be an element of an {\bf ALNN} or {\bf DNN}, as defined in \autoref{de:affinenns}. Therefore $\bf p$ can be parameterized by a vector $\vp$ as in \autoref{eq:p} (for shallow linear neural networks) or \autoref{eq:pd} (for deep neural networks). For {\bf GQNN}s the adequate parametrization can be read out from \autoref{de:quadratic}.
	In other words, we can write every searched for function $\bf p$ via an operator $\Psi$ which maps a parameter $\vp \in \R^{n_*}$ to a function in ${\bf X}$, i.e.,
	\begin{equation*}
		\vx \to {\bf p}(\vx) = \Psi[\vp](\vx).
    \end{equation*}
	Therefore, we aim for solving the \emph{nonlinear} operator equation
	\begin{equation} \label{eq:sol}
		\mathcal{N}(\vp):=F \circ \Psi[\vp] = \by.
	\end{equation}
	For {\bf ALNN}s  we have proven in \cite{SchHofNas23} that the Gauss-Newton method (see \autoref{eq:newton}, below) is locally, quadratically convergent under conditions, which guarantee that during the iteration the gradients of $\Psi$ does not degenerate, i.e., that ${\bf P}$ is a finite-dimensional manifold. Let us now calculate derivatives of the radial neural network operators $\Psi$ (see \autoref{eq:radial_approximation}), which are the basis to prove that also the Gauss-Newton with {\bf RQNN}s is convergent:
	\begin{example} \label{ex:quadratic}
	Let $\sigma : \R \to \R$ be continuous activation function, where all derivatives up to order $2$ are uniformly bounded, i.e., in formulas
	\begin{equation} \label{eq:sigma}
		\sigma \in C^2(\R;\R) \cap \mathcal{B}^2(\R;\R),
    \end{equation}
   such as the sigmoid function.
	 Then, the derivatives of a radial neural network ({\bf RQNN}) $\Psi$ as in \autoref{eq:radial_approximation}
	with respect to the coefficients of $\vp$ in \autoref{eq:pq} are given by the following formulas, where $\nu_s$ is defined in \autoref{eq:nus}:
	\begin{itemize}
		\item Derivative with respect to $\alpha_s$, $s=1,\ldots,N$:
		\begin{equation}\label{eq:ca_d1}
			\begin{aligned}
				\frac{\partial \Psi}{\partial \alpha_s}[\vp](\vx) &= \sigma (\nu_s) \text{ for } s=1,\ldots,N.
			\end{aligned}
		\end{equation}	
		\item Derivative with respect to $w_s^{(t)}$ where $s=1,\ldots,N$, $t=1,\ldots,n$:
		\begin{equation}\label{eq:ca_d2}
			\begin{aligned}
				\frac{\partial \Psi}{\partial w_s^{(t)}} [\vp](\vx) &= \sum_{j=1}^{N} \alpha_j\sigma' (\nu_j) \delta_{s=j}x_t = \alpha_s \sigma' (\nu_s) x_t \;.
			\end{aligned}
		\end{equation}
		\item Derivative with respect to $\theta_s$ where $s=1,\ldots,N$:
		\begin{equation}\label{eq:ca_d3}
			\begin{aligned}\frac{\partial \Psi}{\partial \theta_s} [\vp](\vx)&= \sum_{j=1}^{N} \alpha_j\sigma' (\nu_j) \delta_{s=j} = \alpha_s \sigma' (\nu_s).
			\end{aligned}
		\end{equation}
	    \item Derivative with respect $\xi_s$:	
	    \begin{equation}\label{eq:quadratic_approximation_2}
	    	\frac{\partial \Psi}{\partial \xi_s}[\vp](\vx) = \alpha_s \sigma' (\nu_s) \norm{\vx}^2.
	    \end{equation}
	
	\end{itemize}
	Note, that all the derivatives above are functions in ${\bf X} = L^2([0,1]^n)$.
\end{example}
	For formulating a convergence result for the Gauss-Newton method (see \autoref{eq:newton}) below) we need to specify the following assumptions, which were postulated first in \cite{SchHofNas23} to prove local convergence of a
	Gauss-Newton method on a set of shallow affine linear network functions. Here we verify convergence of the Gauss-Newton method on a set of {\bf RQNN}s, as defined in \autoref{eq:quadratic_approximation}. The proof is completely analogous as in \cite{SchHofNas23}.
	
    \begin{assumption}\label{ass:general}
    	\begin{itemize}
    		\item $F:{\bf X} = L^2([0,1]^n) \to {\bf Y}$ be a linear, bounded operator with trivial nullspace and dense range.
    		\item  $\Psi: \domain{\Psi} \subseteq \R^{(n+2)N} \to {\bf X}$ is a shallow {\bf RQNN} generated by a strictly monotonic activation function $\sigma$ which satisfies \autoref{eq:sigma},
    		like a sigmoid function.
    		\item All derivatives in \autoref{eq:ca_d1} - \autoref{eq:quadratic_approximation_2} are locally linearly independent functions.
    \end{itemize}
    \end{assumption}
Now, we recall a local convergence result:
\begin{theorem}[Local convergence of Gauss-Newton method with {\bf RQNN}s] \label{th:newtonNN}
	Let $\mathcal{N}$ be as in \autoref{eq:sol} be the composition of a linear operator $F$ and the {\bf RQNN} network
	operator as defined in \autoref{eq:radial_approximation}, which satisfy \autoref{ass:general}.
	Moreover,
	let $\vp^0 \in \domain{\Psi}$, which is open, be the starting point of the Gauss-Newton iteration
			\begin{equation} \label{eq:newton} \begin{aligned}
			\vp^{k+1} = \vp^k - \mathcal{N}'(\vp^k)^\dagger(\mathcal{N}(\vp^k)-\by)
			\quad k \in \N_0,
		\end{aligned}
	\end{equation}
    where $\mathcal{N}'(\vp^k)^\dagger$ denotes the Moore-Penrose inverse of $\mathcal{N}'(\vp^k)$.\footnote{For a detailed exposition on generalized inverses see \cite{Nas76}.}

    Moreover,
	let $\vp^\dagger \in \domain{\Psi}$ be a solution of \autoref{eq:sol}, i.e.,
	\begin{equation}\label{eq:h}
	  \mathcal{N}(\vp^\dagger) = \by.
	\end{equation}
	Then the Gauss-Newton iterations are locally, that is if $\vp^0$ is sufficiently close to $\vp^\dagger$, and quadratically converging.
\end{theorem}
\begin{remark}
	The essential condition in \autoref{th:newtonNN} is that all derivatives of $\Psi$ with respect to $\vp$ are linearly independent, which is a nontrivial research question. In \cite{SchHofNas23} we studied convergence of the Gauss-Newton method on a set of shallow affine linear neural networks, where the convergence result required linear independence of the derivatives specified in
	\autoref{eq:ca_d1} - \autoref{eq:ca_d3}. Here, for {\bf RQNN}s, on top, linear independence of the 2nd order moment function \autoref{eq:quadratic_approximation_2} needs to hold. Even linear independence of neural network functions (with derivatives) is still a challenging research topic (see \cite{Lam22}).
	
    Under our assumptions the ill--posedness of $F$ does not affect convergence of the Gauss-Newton method.
    	The generalized inverse completely annihilates $F$ in the iteration. The catch however is, that the data has to be attained in a finite dimensional space spanned by neurons (that is \autoref{eq:h} holds). This assumption is in fact restrictive: Numerical tests show that convergence of Gauss-Newton methods becomes arbitrarily slow if it is violated. Here the ill-posedness of $F$ enters. The proof of \autoref{th:newtonNN} is based on the \emph{affine covariant} condition (see for instance \cite{DeuHoh91}).
\end{remark}

In the following we show the complexity of the convergence condition of a Gauss-Newton method in the case when affine linear
\emph{deep neural network functions} are used for coding.
\begin{example}[Convergence conditions of deep neural networks] \label{ex:ccdn}
	We restrict our attention to the simple case $n=1$. We consider a 4-layer {\bf DNN} (consisting of two internal layers) with $\sigma_1=\sigma_2=\sigma$, which reads as follows:
	\begin{equation}\label{eq:DNN1d}
		x \to \Psi[\vp](x) := \sum_{j_2=1}^{N_2} \alpha_{j_2,2}
		\sigma \left( w_{j_2,2}
		\left( \sum_{j_1=1}^{N_1} \alpha_{j_1,1} \sigma \left(w_{j_1,1} x +\theta_{j_1,1} \right)\right) + \theta_{j_2,2}\right) .
	\end{equation}
    Now, we calculate by chain rule $\frac{\partial \Psi}{\partial w_{1,1}}[\vp](x)$. For this purpose we define
    \begin{equation*}
    	w_{1,1} \to \rho(w_{1,1}) := \sum_{j_1=1}^{N_1} \alpha_{j_1,1} \sigma \left(w_{j_1,1} x +\theta_{j_1,1} \right).
    \end{equation*}
    With this definition we rewrite \autoref{eq:DNN1d} to
    \begin{equation*}
    	x \to \Psi[\vp](x) := \sum_{j_2=1}^{N_2} \alpha_{j_2,2}
    	\sigma \left( w_{j_2,2} \rho(w_{1,1}) + \theta_{j_2,2}\right) ,
    \end{equation*}
    Since
    \begin{equation*}	
    	\rho'(w_{1,1}) := \alpha_{j_1,1} \sigma' \left(w_{1,1} x +\theta_{1,1} \right)x,
    \end{equation*}
    we therefore get
    \begin{equation}\label{eq:DNN1da}\begin{aligned}
	\frac{\partial \Psi}{\partial w_{1,1}} &= \sum_{j_2=1}^{N_2} \alpha_{j_2,2}
	\sigma' \left( w_{j_2,2} \rho(w_{1,1}) + \theta_{j_2,2}\right) w_{j_2,2} \rho'(w_{1,1}) \\
	&= \alpha_{j_1,1} x \sum_{j_2=1}^{N_2} \alpha_{j_2,2} w_{j_2,2}	
	\sigma' \left( w_{j_2,2} \rho(w_{1,1}) + \theta_{j_2,2}\right)  \sigma' \left(w_{1,1} x +\theta_{1,1} \right).
	\end{aligned}
\end{equation}
Note that $\Psi[\vp]$ is a function of $x$ depending on the parameter $\vp$, which contains $w_{1,1}$ as one component.

The complicating issue in a potential convergence analysis concerns the last identity, \autoref{eq:DNN1da}, where evaluations of $\sigma'$ at different arguments appear simultaneously, which makes it complicated to analyze and interpret linear independence of derivatives of $\Psi$ with respect to the single elements of $\vp$. Note that for proving convergence of the Gauss-Newton method, according to \autoref{eq:DNN1da}, we need to verify linear independence of then functions
	\begin{equation*}
		\begin{aligned}
			x \to&  \sigma' \left( w_{j_2,2} \rho(w_{1,1}) + \theta_{j_2,2}\right)  \sigma' \left(w_{1,1} x +\theta_{1,1} \right) \\
			& = \sigma' \left( w_{j_2,2} \left(\sum_{j_1=1}^{N_1} \alpha_{j_1,1} \sigma \left(w_{j_1,1} x +\theta_{j_1,1}  + \theta_{j_2,2}\right) \right)\right)\sigma' \left(w_{1,1} x +\theta_{1,1} \right).
		\end{aligned}
	\end{equation*}
    In other word, the potential manifold of quadratic neural networks is extremely complex.
\end{example}

The conclusion of this section is that the analysis of iterative regularization methods, like a Gauss-Newton method, gets significantly more complicated if deep neural networks are used as ansatz functions. In contrast using neural network with higher order nodes (such as radial) results in transparent moment conditions. So the research question discussed in the following is whether shallow higher order neural networks have similar approximation properties than deep neural network functions, which would reveal clear benefits of such for analyzing iterative methods for solving ill--posed problems.

	\section{Convergence rates for universal approximation of {\bf RQNN}s}
	\label{sec:rates}
	In the following we prove convergence rates of {\bf RQNN}s (as defined in \autoref{eq:radial_approximation}) in the $\mathcal{L}^1$-norm. To be precise we specify a subclass on {\bf RQNN}s for which we prove convergence rates results. This is a much finer result than the standard universal approximation result, \autoref{th:general}, since it operates on a subclass and provides convergence rates. However, it is the only approximation result so far, which we can provide. We recall that the constraints, $\xi_j \neq 0$ and $\kappa_j/\xi_j < 0$ (see \autoref{eq:kappa}), induce that the level-sets of the neurons are compact (circles) and not unbounded, as they are for shallow affine linear neurons. Therefore we require a different analysis (and different spaces) than in the most advanced and optimal convergence rates results for affine linear neural networks as for instance in \cite{SieXu22,SieXu23}. In fact the analysis is closer to an analysis of compact wavelets.

    The convergence rate will be expressed in the following norm:
 \begin{definition}\label{de:ell1space} $\mathcal{L}^1$ denotes the space of square integrable functions on $[0,1]^n$, which satisfy \cite[(1.10)]{BarCohDahDev08}
 		\begin{equation*}
 			\|f\|_{\mathcal{L}^1} := \inf\set{\sum_{g\in D}|c_g| : f= \sum_{g\in D}c_g g} < \infty.
 		\end{equation*}
 		Here $c_g$ are the coefficients of an wavelet expansion and $D$ is a countable, general set of wavelet functions.
        A discrete wavelet basis (see, for example, \cite{Dau92,Gra95,LouMaaRie98,Chu92}) is a (not necessarily orthonormal)
        basis of the space of square integrable functions, where the basis elements are given by dilating and 
        translating a mother wavelet function.
 \end{definition}
\begin{remark}
 Notice that the notation $\mathcal{L}^1$ does not refer to the common $L^1$-function space of absolutely integrable functions and depends on the choice of the wavelet system. For more properties and details on this space see \cite[Remark 3.11]{ShaCloCoi18}.
\end{remark}

	\begin{definition}[Circular scaling function and wavelets] \label{de:cf}
		Let $r > 0$ be a fixed constant and $\sigma$ be a discriminatory function as defined in \autoref{de:disc_function} such that $\abs{\int_{\R^n}\sigma(r^2-\|\vx\|^2)d\vx} < \infty$. Then let
		\begin{equation}\label{eq:varphi}
			\vx \in \R^n \to \varphi(\vx):= C_n \sigma(r^2 - \norm{\vx}^2),
		\end{equation}
		where $C_n$ is a normalizing constant such that $\int_{\R^n}\varphi(\vx)d\vx=1$.\\	
		Then we define for $k \in \Z$ the \emph{radial scaling functions} and \emph{wavelets}
		\begin{equation}\label{eq:S}
			\begin{aligned}
				(\vy,\vy) \in \R^n \times \R^n \to \scinn{k}(\vx,\vy) &:= 2^k \varphi(2^{k/n}(\vx-\vy)) \text{ and } \\
				(\vy,\vy) \in \R^n \times \R^n \to \wcinn{k}(\vx,\vy) &:= 2^{-k/2} (\scinn{k}(\vx,\vy)- \scinn{k-1}(\vx,\vy)).
			\end{aligned}
		\end{equation}
        Often $\vy \in \R^n$ is considered a parameter and we write synonymously
			\begin{equation}\label{eq:Sa}
		\begin{aligned}
			\vx \to \scinn{k,\vy}(\vx) = \scinn{k}(\vx,\vy) \text{ and }
			\vx \to \wcinn{k,\vy}(\vx) = \wcinn{k}(\vx,\vy).
		\end{aligned}
	\end{equation}
    In particular, this notation means that $\scinn{k,\vy}$ and $\wcinn{k,\vy}$ are considered solely functions of the variable $\vx$.

	We consider the following subset of {\bf RQNN}s
	\begin{equation} \label{eq:fpcinnd}
		\mathcal{S}^C_d := \set{S^C_{k,\vy}: k \in \Z \text{ and } \vy \in 2^{-k/n}\Z^n}.
	\end{equation}
This is different to standard universal approximation theorems, where an uncountable number of displacements are considered.
	Moreover, we define the discrete Wavelet space         	
	\begin{equation} \label{eq:fpwinnd}
		\mathcal{W}^C_d :=
		\set{\wcinn{k,\vy} := 2^{-k/2} \left(\scinn{k,\vy}-\scinn{k-1,\vy}\right) :
			k \in \Z \text{ and } \vy \in 2^{-k/n}\Z^n}.
	\end{equation}
\end{definition}

		According to \autoref{le:WaveletApprox}, in order to prove the following approximation result, we only require to verify that $\mathcal{S}_d^C$ satisfies the \textcolor{blue}{conditions for} symmetric AtI and the double Lipschitz condition (see \autoref{def:wavelet}), which originate from \cite[Def. 3.4]{DenHan09}.

		\begin{corollary}$\mathcal{W}_d^C$
			is a frame and for every function $f \in \mathcal{L}^1(\R^n)$ there exists a linear combination of $N$ elements of $\mathcal{W}_d^C$, denoted by $f_N$, satisfying
			\begin{equation} \label{eq:app_general_c}
				\norm{f-f_N}_{L^2} \leq \norm{f}_{\mathcal{L}^1} (N+1)^{-1/2}.
			\end{equation}
	\end{corollary}
	For proving the approximation to identity, AtI, property of the functions $\scinn{k}$, $k \in \Z$, we use the following basic inequality.
	\begin{lemma}\label{lem:hessian}
		Let $h: \R^n \to \R$ be a twice differentiable function, which can be expressed in the following way:
		\begin{equation*}
			h(\vx)=h_s(\norm{\vx}^2) \text{ for all } \vx \in \R^n.
		\end{equation*}
		Then the spectral norm of the Hessian of $h$ can be estimated as follows:\footnote{In the following $\nabla$ and $\nabla^2$ (without subscripts) always denote derivatives with respect to an $n$-dimensional variable such as $\vx$. ${}'$ and ${}''$ denotes derivatives of a one-dimensional function.}
		\begin{equation} \label{eq:hessian}
			\norm{\nabla^2 h(\vx)} \leq \max \set{\abs{4\norm{\vx}^2 h_s''(\norm{\vx}^2)+ 2h_s'(\norm{\vx}^2)},\abs{2h_s'(\norm{\vx}^2)}}.
		\end{equation}
	\end{lemma}
	\begin{proof}
		Since $\nabla^2 h(\vx)$ is a symmetric matrix, its operator norm is equal to its spectral radius, namely the largest absolute value of an eigenvalue. By routine calculation we can see that
		\[
		\nabla_{x_i x_j}h(\vx)=4x_ix_jh_s''(\norm{\vx}^2) + 2\delta_{ij}h_s'(\norm{\vx}^2).
		\]
		\begin{equation*}
			4 h_s''(\norm{\vec{x}}^2) C \vz = (-2 h_s'(\norm{\vx}^2) + \lambda)\vz.
		\end{equation*}
		Or in other words $\frac{- 2 h_s'(\norm{\vx}^2) + \lambda}{4 h_s''(\norm{x}^2)}$ is an eigenvalue of $C$. Moreover, $C = \vx \vx^T$ is a rank one matrix and thus the spectral values are $0$ with multiplicity $(n-1)$ and $\norm{\vx}^2$. This in turn shows that the eigenvalues of the Hessian
		are $+2h_s'(\norm{\vx}^2)$ (with multiplicity $n-1$) and $4\norm{\vx}^2 h_s''(\norm{\vx}^2)+2h_s'(\norm{\vx}^2)$, which proves \autoref{eq:hessian}.
		
	\end{proof}
	
	In the following lemma, we will prove that the kernels $(\scinn{k})_{k \in \Z}$ are an AtI.
	\begin{lemma}\label{le:SConditions} Let $r>0$ be fixed.
		Suppose that the activation function $\sigma:\R \to \R$ is monotonically increasing
		and satisfies for the $i$-th derivative ($i=0,1,2$)
		\begin{equation} \label{eq:sigma0}
			\abs{\sigma^i (r^2-t^2)} \leq C_\sigma (1+\abs{t}^n)^{-1-{(2i+1)/n}} \text{ for all } t \in \R.
		\end{equation}
		Then the kernels $(\scinn{k})_{k \in \Z}$ as defined in \autoref{eq:S} form an AtI as defined in \autoref{def:wavelet} that also satisfy the double Lipschitz condition  \autoref{eq:DoubleLipschitzCondition}.
	\end{lemma}
	
	\begin{proof}
		We verify the three conditions from \autoref{def:wavelet} as well as \autoref{eq:DoubleLipschitzCondition}.
		First of all, we note that
		\begin{equation} \label{eq:sigma1}
			\abs{\sigma^i (r^2 - \norm{\vx}^2)} \leq C_\sigma (1+\norm{\vx}^n)^{-1-{(2i+1)/n}} \text{ for all } \vx \in \R^n.
		\end{equation}
		\begin{itemize}
			\item \textbf{Verification of \autoref{it1} in \autoref{def:wavelet}:}
			\autoref{eq:varphi} and \autoref{eq:sigma0} imply that
			\begin{equation}\label{eq:VarphiBound}
				0 \leq \varphi(\vx-\vy) = C_n \sigma(r^2-\norm{\vx-\vy}^2) \leq C_\sigma C_n (1+\norm{\vx-\vy}^n)^{-1-1/n} \text{ for all } \vx,\vy \in \R^n.
			\end{equation}
			Therefore
			\begin{equation*} \begin{aligned}
					\scinn{k}(\vx,\vy)&= 2^k \varphi(2^{k/n}(\vx-\vy))
					\leq C_\sigma C_n 2^{k}(1+2^k \norm{\vx-\vy}^n)^{-1-1/n} \\
					&= C_\sigma C_n 2^{-k/n}(2^{-k}+ \norm{\vx-\vy}^n)^{-1-1/n}.
			\end{aligned} \end{equation*}
			Thus \autoref{it1} in \autoref{def:wavelet} holds with $\epsilon=1/n$ and $C_\rho=1$ and $C=C_n C_\sigma$.
			\item \textbf{Verification of \autoref{it2} in \autoref{def:wavelet} with $C_\rho=1$ and $C_A=2^{-n}$:}
			Because $\sigma$ is monotonically increasing it follows from \autoref{eq:varphi} and the fact that $S_{0}(\vec{x}, \vec{y})=\varphi(\vec{x}-\vec{y})$  (see \autoref{eq:S}) and \autoref{de:cf} that
			$$F_\vy(\vx):= \norm{\nabla_\vx (S_0(\vx,\vy))} = 2 C_n \norm{\vx-\vy}\sigma'(r^2-\norm{\vx-\vy}^2) \text{ for all } \vy \in \R^n.$$
			Then \autoref{eq:sigma1} implies that
			\begin{equation*}
				\begin{aligned}
					F_\vy(\vx) &\leq 2C_n C_\sigma (1+\norm{\vx-\vy}^n)^{-1-3/n} \norm{\vx-\vy} \\& \leq
					2C_n C_\sigma (1+\norm{\vx-\vy}^n)^{-1-3/n} (1+\norm{\vx-\vy}^n)^{\frac{1}{n}}\\
					& = 2C_n C_\sigma (1+\norm{\vx-\vy}^n)^{-1-{2/n}}.
				\end{aligned}
			\end{equation*}
			From the definition of $\scinn{k}(\vx,\vy)$, it follows
			\begin{equation} \label{eq:zw}
				\begin{aligned}
					\norm{\nabla_\vx (\scinn{k}(\vx,\vy))} &= \norm{\nabla_\vx (2^k\varphi(2^{{k/n}}(\vx-\vy)))} =
					2^k \norm{\nabla_\vx S_0 (2^{{k/n}}\vx,2^{{k/n}}\vy)} \\
					&= 2^{k+{k/n}}F_{2^{{k/n}}\vy} (2^{{k/n}}\vx)\\
					&\leq 2^{-k/n} C_n C_\sigma (2^{-k}+\norm{\vx-\vy}^n)^{-1-{2/n}}.
				\end{aligned}
			\end{equation}
			From the mean value theorem it therefore follows from \autoref{eq:zw} and \autoref{eq:kk} that
			\begin{equation} \label{eq:help}
				\begin{aligned}
					~ & \frac{\abs{\scinn{k}(\vx,\vy) - \scinn{k}(\vx',\vy)}}{\norm{\vx-\vx'}} \\
					\leq &
					\max_{\set{\vz = t \vx' + (1-t)\vx:t \in [0,1]}} \norm{\nabla_\vx(\scinn{k}(\vz,\vy))}\\
					\leq &2^{-k/n} C_n C_\sigma \max_{\set{\vz = \vx + t (\vx' -\vx):t \in [0,1]}} (2^{-k}+\norm{\vz-\vy}^n)^{-1-{2/n}} \\
					= & 2^{-k/n} C_n C_\sigma \left(2^{-k} + \min_{\set{\vz = \vx + t (\vx' -\vx):t \in [0,1]}} \norm{\vz-\vy}^n \right)^{-1-{2/n}}.
				\end{aligned}
			\end{equation}
			Then application of \autoref{eq:kk} and noting that $\frac{1-2^{-n}}{2^{-n}} \geq 1$
			gives
			\begin{equation*}
				\begin{aligned}
					~ & \frac{\abs{\scinn{k}(\vx,\vy) - \scinn{k}(\vx',\vy)}}{\norm{\vx-\vx'}} \\
					\leq &
					2^{-k/n} C_n C_\sigma \left((1-2^{-n}) 2^{-k}+2^{-n}\norm{\vx-\vy}^n  \right)^{-1-{2/n}}\\ \leq &
					2^{-k/n} \left(2^{-n}\right)^{-1-{2/n}}
					C_n C_\sigma \left(\frac{(1-2^{-n}) }{2^{-n}}2^{-k} + \norm{\vx-\vy}^n  \right)^{-1-{2/n}}\\
					\leq &
					2^{-k/n} 2^{n+2}
					C_n C_\sigma \left(2^{-k} + \norm{\vx-\vy}^n  \right)^{-1-{2/n}}.
				\end{aligned}
			\end{equation*}
			Therefore \autoref{it2} is satisfied with $C_\rho=1$, $\zeta=1/n$, $\epsilon=1/n$ and
			$C=2^{n+2} C_n C_\sigma $.
			
			\item \textbf{Verification of \autoref{it4} in \autoref{def:wavelet}:}
			From the definition of $\scinn{k}$ (see \autoref{eq:S}) it follows that for every $k\in \Z$ and $\vy \in \R^n$
			\begin{equation*}
				1 = \int_{\R^n}\scinn{k}(\vx,\vy)d\vx = \int_{\R^n} 2^k \varphi(2^{k/n}(\vx-\vy)) d\vx.
			\end{equation*}
			\item \textbf{Verification of the double Lipschitz condition \autoref{eq:DoubleLipschitzCondition} in \autoref{def:wavelet}:}
			By using the integral version of the mean value theorem, we have
			\begin{align*}
				&\scinn{k}(\vx,\vy) - \scinn{k}(\vx',\vy) - \scinn{k}(\vx,\vy') + \scinn{k}(\vx',\vy') \\
				&= \scinn{k}(\vx,\vy) - \scinn{k}(\vx',\vy) - (\scinn{k}(\vx,\vy') - \scinn{k}(\vx',\vy'))\\
				&= \int_{0}^{1} \langle \nabla_\vy \scinn{k}(\vx,\vy' + t(\vy-\vy')), \vy-\vy' \rangle dt \\
				& \qquad - \int_{0}^{1} \langle \nabla_\vy \scinn{k}(\vx',\vy' + t(\vy-\vy')), \vy-\vy' \rangle dt\\
				&= \int_{0}^{1}\int_{0}^{1} \langle \nabla_{\vx, \vy} \scinn{k}(\vx' + s(\vx-\vx'),\vy' + t(\vy-\vy'))(\vx-\vx'), \vy-\vy' \rangle dt ds.
			\end{align*}
			Following this identity, we get
			\begin{equation} \label{eq:lip2}
				\begin{aligned}
					~ & \frac{\abs{\scinn{k}(\vx,\vy) - \scinn{k}(\vx',\vy) + \scinn{k}(\vx,\vy') - \scinn{k}(\vx',\vy')}}
					{\norm{\vx-\vx'} \norm{\vy-\vy'}}\\
					\leq &  \max_{\interval}
					\norm{\nabla_\vy \left(\frac{\scinn{k}(\vx,\vz) - \scinn{k}(\vx',\vz)}{\norm{\vx-\vx'}}\right)}
					\leq  \max_{\quader} \norm{\nabla^2_{\vx\vy}\scinn{k}(\vz',\vz)},
				\end{aligned}
			\end{equation}
			where
			\begin{equation*}
				\begin{aligned}
					\interval &:= \set{\vz=t\vy+(1-t)\vy':t \in [0,1]}, \\
					\quader &:= \set{(\vz=t_{\vy}\vy+(1-t_{\vy})\vy',\vz'=t_{\vx}\vx+(1-t_{\vx})\vx'):
						t_{\vy} \in [0,1],t_{\vx} \in [0,1]},
				\end{aligned}
			\end{equation*}
			and $\norm{\nabla^2_{\vx\vy}\scinn{k}(\vz',\vz)}$ denotes again the spectral norm of $\nabla^2_{\vx\vy}\scinn{k}(\vz',\vz)$.
			
			Now, we estimate the right hand side of \autoref{eq:lip2}: From the definition of $\scinn{k}$, \autoref{eq:S}, and the definition of $\varphi$, \autoref{eq:varphi}, it follows with the abbreviation
			$\vec{\omega} = 2^{{k/n}}(\vz'-\vz)$:
			\begin{equation*}
				\begin{aligned}
					\norm{\nabla^2_{\vx\vy}\scinn{k}(\vz',\vz)} &=
					2^{k}\norm{\nabla^2_{\vx\vy}(\varphi \circ (2^{{k/n}}\cdot))(\vz'-\vz))}
					= 2^{k(1+2/n)}\norm{\nabla^2 \varphi (\vec{\omega}))}.
				\end{aligned}
			\end{equation*}
			Applications of \autoref{lem:hessian} with $\vx \to h(\vx)=\varphi(\vx)$ and $$t \to h_s(t) = C_n \sigma(r^2-t)$$ shows that (note that $h_s'(t)=-C_n \sigma'(r^2-t)$)
			\begin{equation}\label{eq:last}
				\begin{aligned}
					&\norm{\nabla^2 \varphi (\vec{\omega})} \\
					\leq & C_n  \max \set{\abs{4\norm{\vec{\omega}}^2 \sigma''(r^2 - \norm{\vec{\omega}}^2) -  2\sigma'(r^2 - \norm{\vec{\omega}}^2)},\abs{2\sigma'(r^2 - \norm{\vec{\omega}}^2)}}\\
					\leq & 2^2 C_n  \max \set{2\norm{\vec{\omega}}^2 \abs{\sigma''(r^2 - \norm{\vec{\omega}}^2)},\abs{\sigma'(r^2 - \norm{\vec{\omega}}^2)}}.
				\end{aligned}
			\end{equation}
			Thus from \autoref{eq:sigma1} it follows that
			\begin{equation*}
				\begin{aligned}
					\norm{\nabla^2_{\vx\vy}\scinn{k}(\vz',\vz)}
					\leq & 2^2 2^{k(1+2/n)}C_n C_\sigma  \max \set{2\norm{\vec{\omega}}^2 (1+\norm{\vec{\omega}}^n)^{-1-{5/n}}, (1+\norm{\vec{\omega}}^n)^{-1-3/n} }\\
					\leq & 2^3 2^{k(1+2/n)}C_n C_\sigma  (1+\norm{\vec{\omega}}^n)^{-1-3/n}\\
					\leq & 2^{-k/n} 2^3 C_n C_\sigma  (2^{-k} + \norm{ \vz'-\vz}^n)^{-1-3/n} .
				\end{aligned}
			\end{equation*}
			In the next step we note that from \autoref{eq:kk1} it follows that
			\begin{equation}\label{eq:min}
				\norm{\vz'-\vz}^n \geq 3^{-n} \norm{\vx-\vy}^n - 3^{-n} 2^{1-k}.
			\end{equation}
			
			Thus we get because $\frac{1-3^{-n}2}{3^{-n}} \geq 1$
			\begin{equation*}
				\begin{aligned}
					~ & \frac{\abs{\scinn{k}(\vx,\vy) - \scinn{k}(\vx',\vy) - \scinn{k}(\vx,\vy')  + \scinn{k}(\vx',\vy')}}
					{\norm{\vx-\vx'} \norm{\vy-\vy'}} \\
					\leq & 2^{-k/n} 2^3 C_n C_\sigma \left((1-3^{-n}2) 2^{-k} + 3^{-n}\norm{\vx-\vy}^n \right)^{-1-3/n}\\
					\leq & 2^{-k/n} 2^3 (3^{-n})^{-1-3/n} C_n C_\sigma
					\left(\frac{1-3^{-n}2}{3^{-n}} 2^{-k} + \norm{\vx-\vy}^n \right)^{-1-3/n}\\
					\leq & 2^{-k/n} 2^3 3^{n+3} C_n C_\sigma
					\left(2^{-k} + \norm{\vx-\vy}^n \right)^{-1-3/n}.
				\end{aligned}
			\end{equation*}
			Therefore \autoref{it4} is satisfied with $C_\rho=1$,
			$\tilde{C} = 2^3 3^{n+3} C_n C_\sigma$, $\zeta=1/n$, and $\epsilon=1/n$.
		\end{itemize}
	\end{proof}

		We have shown with \autoref{le:SConditions} that the functions $\set{\scinn{k}: k \in \Z}$ form an AtI, which satisfies the double Lipschitz condition (see \autoref{def:wavelet}). 
        Moreover, for activations functions, like the sigmoid function, \autoref{eq:sigma0}
        it follows that 
		\begin{equation} \label{eq:sigma0p}
			\lim_{t \to \pm \infty} (1+\abs{t}^n)^{1+{(2i+1)/n}} \abs{\sigma^i (r^2-t^2)} = 0,
		\end{equation}
         which shows \autoref{eq:sigma}. The limit result is an easy consequence of the fact that $t\to \e^{-t^2}$ and its derivative are faster decaying to $zero$ than the functions $t \to (1+\abs{t}^n)^{-1-{(2i+1)/n}}$.
Therefore, it follows from \autoref{le:WaveletApprox} that the set $\mathcal{W}_d^C$ is a wavelet frame and satisfies the estimate
		\autoref{eq:app_general}. We summarize this important result:

	\begin{theorem}[$\mathcal{L}^1$-convergence of radial wavelets] \label{thm:ConvR}
		Let $\sigma$ be an activation function that satisfies the conditions in \autoref{le:SConditions}.
		Then, for every function $f\in \mathcal{L}^1(\R^n)$ (see \autoref{de:ell1space} for a definition of the space) and every
		$N \in \N$, there exists a function
		\begin{equation*}
				f_N \in \text{span}_N(\mathcal{W}^C_d) \subseteq \mathcal{L}^1(\R^n),
		\end{equation*}
	 and $\text{span}_N$ denotes linear combinations of at most $N$ terms in the set, such that
		\begin{equation}
			\label{eq:app_error}
			\norm{f-f_N}_{L^2} \leq \norm{f}_{\mathcal{L}^1} (N+1)^{-1/2}.
		\end{equation}
	\end{theorem}


     Using \autoref{thm:ConvR} we are able to formulate the main result of this paper:
	\begin{corollary}[$\mathcal{L}^1$-convergence of {\bf RQNN}s] \label{co:ConvR}
	Let the assumptions of \autoref{thm:ConvR} be satisfied.
	Then, for every function $f\in \mathcal{L}^1(\R^n)$ (see \autoref{de:ell1space}) and every
	$N \in \N$, there exists a parametrization vector
	\begin{equation*}
		\vec{p} = [\alpha_1,\ldots,\alpha_{2N};{\bf w}_1,\ldots, {\bf w}_{2N}; \xi_1,\ldots,\xi_{2N};\theta_1,\ldots,\theta_{2N}].
	\end{equation*}
	such that
	\begin{equation}
		\label{eq:app_error II}
		\norm{f-\Psi[\vp]}_{L^2} \leq \norm{f}_{\mathcal{L}^1} (N+1)^{-1/2},
	\end{equation}
	where $\Psi[\vp]$ is a {\bf RQNN} from \autoref{eq:radial_approximation}.
\end{corollary}

    \begin{proof}
    	From \autoref{thm:ConvR} it follows that there exists $f_N \in \text{span}_N(\mathcal{W}^C_d)$, which satisfies \autoref{eq:app_error}. By definition,
    	\begin{equation*}
    		f_N(\vx) = \sum_{k \in \Z} \sum_{\vec{j} \in \Z^n} \chi_{f,N} (k,\vec{j}) \beta_{k,\vec{j}}
    		\wcinn{k,2^{-k/n}\vec{j}}(\vx),
    	\end{equation*}
    	where $\chi_{f,N}$ denotes the characteristic function of a set of indices of size $N$.	
    	Using the definition of $\wcinn{k}$ we get from \autoref{eq:S} and \autoref{eq:varphi}
    	\begin{equation*}\begin{aligned}
    		f_N(\vx) &= \sum_{k \in \Z} \sum_{\vec{j} \in \Z^n} \chi_{f,N} (k,\vec{j}) \beta_{k,\vec{j}} 2^{-k/2}
    		\left(\scinn{k,2^{-k/n}\vec{j}}(\vx)- \scinn{k-1,2^{-k/n}\vec{j}}(\vx)\right)\\
    	&= \sum_{k \in \Z} \sum_{\vec{j} \in \Z^n} \chi_{f,N} (k,\vec{j}) \beta_{k,\vec{j}} 2^{-k/2}
    	\left( 2^k \varphi(2^{k/n}(\vx-2^{-k/n}\vec{j})) -
    	2^{k-1} \varphi(2^{(k-1)/n}(\vx-2^{-k/n}\vec{j})) \right)\\
    	&= \sum_{k \in \Z} \sum_{\vec{j} \in \Z^n} \chi_{f,N} (k,\vec{j}) C_n \beta_{k,\vec{j}} 2^{-k/2} \\
    	& \qquad \left( 2^k \sigma \left(r^2 - \norm{2^{k/n} \left(\vx-2^{-k/n}\vec{j}\right)
    	                                            }^2    	
    	                           \right) -
    	2^{k-1} \sigma \left( r^2 - \norm{2^{(k-1)/n}\left(\vx-2^{-k/n}\vec{j}\right)}^2
    	\right)
  \right).
    \end{aligned}
    	\end{equation*}
    The arguments where $\sigma$ is evaluated are in general different, and thus we get     	\begin{equation*}\begin{aligned}
    		f_N(\vx)
    		&= \sum_{k \in \Z} \sum_{\vec{j} \in \Z^n} \underbrace{\chi_{f,N} (k,\vec{j}) C_n \beta_{k,\vec{j}} 2^{k/2} }_{=\alpha_j}
    		\sigma \left(\underbrace{r^2}_{=\theta_j} \underbrace{- 4^{k/n}}_{=\xi_j} \|\vx-\underbrace{2^{-k/n}\vec{j}}_{=\vy_j}\|^2 \right) \\  	
    		& \qquad +\sum_{k \in \Z} \sum_{\vec{j} \in \Z^n} \underbrace{\chi_{f,N} (k,\vec{j}) C_n \beta_{k,\vec{j}} 2^{k/2-1}}_{\alpha_j}
    		\sigma \left( \underbrace{r^2}_{\theta_j}
    		              \underbrace{- 4^{(k-1)/n}}_{=\xi_j}
    		              \|\vx-\underbrace{2^{-k/n}\vec{j}}_{\vy_j}\|^2
    		\right)
    	\end{aligned}
    \end{equation*}
    from which the coefficients can be read out. Note that for the sake of simplicity we have not explicitly stated the exact index in the subscript of the coefficients $(\vec{\alpha},\vec{\theta},\vec{\xi},\by)$, or in other words we use the formal relation $j=(k,\vec{j})$ to indicate the dependencies. This shows the claim if we identify
    $\theta_j = r^2$ and
    		\begin{equation*}
    				\bw^T = 2^{1+{2k/n}}\vy.
    			\end{equation*}
    \end{proof}

	\begin{remark}\label{rem:NumberOfneurons}
Our results in section 4 are proved using ``similar'' methods as in \cite{ShaCloCoi18}, but can not be used directly. Because the main ingredient of the proof is the construction of \emph{localizing} functions as wavelets. This task requires two layers of affine linear decision functions to accomplish, since the linear combinations of functions of the form $\vx \to \varphi(\vx):= C_n \sigma(\vw^T \vx + \theta)$ are not localized; they \emph{cannot} satisfy \autoref{it4} in \autoref{def:wavelet}, as they are not integrable on $\R^n$. In comparison, some quadratic decision functions - such as those used in Section 4 - naturally gives localized functions, and therefore can work as wavelets on their own.
\bigskip\par\noindent

We mention that there exists a zoo of approximation rates results for neural network functions (see for instance \cite{SieXu22,SieXu23}, where best-approximation results in dependence of the space-dimension $n$ have been proven). We have based our analysis on convergence rates results of \cite{ShaCloCoi18} and \cite{DenHan09}. This choice is motivated by the fact that the AtI property from \cite{DenHan09} is universal, and allows a natural extension of their results to quadratic neural network functions. The motivation here is not to give the optimal convergence rates result for quadratic neural networks, but to verify that convergence rates results are possible. We believe that AtI is a very flexible and elegant tool.
\end{remark}

	\section{Conclusion}\label{sec:conclusions}
	In this paper we studied generalized neural networks functions for solving inverse problems.
	We proved a universal approximation theorem and we have proven convergence rates of radial neural networks,
	which are the same order as the classical affine linear neural network functions, but with vastly reduced numbers of elements
	(in particular it spares one layer). The paper presents a proof of concept and thus we restrict attention only to radial neural networks although generalizations to higher order neural networks (such as cubic) is quite straightforward.
	
	We also make the remark about the convergence of the actual training process of a neural network, i.e. optimization of the parameters of each decision function. This is usually done by gradient descent or Gauss-Newton-type methods. From the analysis in \cite{SchHofNas23} the convergence conditions of Newton's method become very complicated for {\bf DNN}s, but are transparent for
	quadratic neural networks, which is fact a matter of the complicated chain rule, when differentiating {\bf DNN}s with respect to the parameters. For such an analysis there is a clear preference to go higher with the degree of polynomials than increasing the number of layers.
	\appendix
\section{Approximation to the identity (AtI)} \label{sec:appendix}
\begin{definition}[Approximation to the identity \cite{DenHan09}] \label{def:wavelet}
	A sequence of \emph{symmetric kernel functions} $(S_k: \R^n \times \R^n  \to \R)_{k\in \Z}$ is said to be an \emph{approximation to the identity (AtI)} if there exist a quintuple $(\epsilon,\zeta,C,C_\rho,C_A)$ of positive numbers satisfying the additional constraints
	\begin{equation} \label{eq:quin}
		0 < \epsilon \leq \frac{1}{n}, 0 < \zeta \leq \frac{1}{n} \text{ and } C_A < 1
	\end{equation}
	the following three conditions are satisfied for all $k \in \Z$:
	\begin{enumerate}
		\item \label{it1}
		$\abs{S_k(\vx,\vy)} \leq C\frac{2^{-k\epsilon}}{\left(2^{-k}+ C_\rho\norm{\vx-\vy}^n \right)^{1+\epsilon}}$ for all $\vx,\vy \in \R^n$;
		\item \label{it2}
		$\abs{S_k(\vx,\vy) - S_k(\vx',\vy)} \leq C \left( \frac{C_\rho\norm{\vx-\vx'}^n }{2^{-k}+ C_\rho\norm{\vx-\vy}^n }\right)^{\zeta}\frac{2^{-k\epsilon}}{\left(2^{-k}+ C_\rho\norm{\vx-\vy}^n \right)^{1+\epsilon}}$ \\
		for all triples $(\vx,\vx',\vy) \in \R^n \times \R^n \times \R^n$ which satisfy
		\begin{equation} \label{eq:rest_item2}
			C_\rho\norm{\vx-\vx'}^n  \leq C_A \left(2^{-k}+ C_\rho\norm{\vx-\vy}^n \right);
		\end{equation}
		\item \label{it4} $\int_{\R^n} S_k(\vx,\vy) d\vy = 1$ for all $\vx \in \R^n$.
	\end{enumerate}
	Moreover, we say that the AtI satisfies the \emph{double Lipschitz condition} if there exist a triple $(\tilde{C},\tilde{C}_A,\zeta)$ of positive constants satisfying
	\begin{equation} \label{eq:quin2}
		\tilde{C}_A < \frac12,
	\end{equation}
	such that for all $k \in \Z$
	\begin{equation}\label{eq:DoubleLipschitzCondition}
		\begin{aligned}
			&\abs{S_k(\vx,\vy) - S_k(\vx',\vy) - S_k(\vx,\vy') + S_k(\vx',\vy')}\\
			\leq & \tilde{C} \left( \frac{C_\rho\norm{\vx-\vx'}^n }{2^{-k}+C_\rho\norm{\vx-\vy}^n} \right)^\zeta
			\left( \frac{C_\rho \norm{\vy-\vy'}^n}{2^{-k}+C_\rho\norm{\vx-\vy}^n }\right)^\zeta \frac{2^{-k \epsilon}}{(2^{-k}+C_\rho\norm{\vx-\vy}^n )^{1+\epsilon}}
		\end{aligned}
	\end{equation}
	for all quadruples $(\vx,\vx',\vy,\vy')\in \R^n \times \R^n \times \R^n \times \R^n$ which satisfy
	\begin{equation} \label{eq:rest_item3}
		C_\rho \max\set{\norm{\vx-\vx'}^n,\norm{\vy-\vy'}^n} \leq \tilde{C}_A \left( 2^{-k}+C_\rho\norm{\vx-\vy}^n\right).
	\end{equation}
\end{definition}

The conditions \autoref{it2} and \autoref{eq:rest_item3} are essential for our analysis. We characterize now geometric properties of these constrained sets:
\begin{lemma} \label{lem:kk}
	\begin{itemize} Let $C_\rho=1$, $C_A=2^{-n}$ and $\tilde{C}_A=3^{-n}$.
		\item Then set of triples $(\vx,\vx',\vy)$ which satisfy \autoref{eq:rest_item2} and for which $\norm{\vx - \vy}^n \geq 2^{-k}$ and all $t \in[0,1]$ satisfy
		\begin{equation} \label{eq:kk}
			\norm{\vx+t(\vx'-\vx)-\vy}^n \geq 2^{-n}\norm{\vx-\vy}^n - 2^{-n} 2^{-k}.
		\end{equation}
		\item
		The set of quadrupels $(\vx,\vx',\vy,\vy')$ which satisfy \autoref{eq:rest_item3}  and for which $\norm{\vx - \vy}^n \geq 2^{-k}$ satisfy
		\begin{equation} \label{eq:kk1}	
			\norm{\vx+t_\vx(\vx'-\vx)-\vy-t_\vy(\vy'-\vy)}^n\geq 3^{-n} \norm{\vx-\vy}^n - 3^{-n} 2^{1-k}
		\end{equation}
		for all $t_\vx,t_\vy\in[0,1]$.
	\end{itemize}
\end{lemma}
\begin{proof}
	\begin{itemize}
		\item With the concrete choice of parameters $C_A$, $C_\rho$ \autoref{eq:rest_item2} reads as follows
		\begin{equation} \label{eq:diff}
			\norm{\vx-\vx'}^n \leq  2^{-n}\left(2^{-k} + \norm{\vx-\vy}^n\right).
		\end{equation}
		Since we assume that $\norm{\vx-\vy}^n \geq 2^{-k}$ it follows from \autoref{eq:diff} that
		\begin{equation*}
			\norm{\vx-\vx'} \leq  2^{-1}\norm{\vx-\vy}.
		\end{equation*}
		In particular $\norm{\vx-\vy}-\norm{\vx-\vx'} \geq 0$.
		
		We apply Jensen's inequality, which states that for $a,b \geq 0$
		\begin{equation} \label{eq:jensen}
			a^n+b^n \geq 2^{1-n}(a+b)^n.
		\end{equation}
		We use $a=\norm{\vx+t(\vx'-\vx)-\vy}$ and $b = \norm{t(\vx'-\vx)}$, which then (along with the triangle inequality) gives
		\begin{equation*}
			\norm{\vx+t(\vx'-\vx)-\vy}^n + \norm{t(\vx'-\vx)}^n \geq 2^{1-n}\left(\norm{\vx-\vy}\right)^n.
		\end{equation*}
		In other words, it follows from \autoref{eq:diff} that
		\begin{equation*}
			\begin{aligned}
				\norm{\vx+t(\vx'-\vx)-\vy}^n & \geq 2^{1-n}\norm{\vx-\vy}^n -  t^n\norm{\vx'-\vx}^n \\
				& \geq 2^{1-n}\norm{\vx-\vy}^n -  \norm{\vx'-\vx}^n \\
				& \geq 2^{1-n}\norm{\vx-\vy}^n - 2^{-n}\left(2^{-k} + \norm{\vx-\vy}^n\right)\\
				&= 2^{-n} \norm{\vx-\vy}^n - 2^{-n} 2^{-k}.
			\end{aligned}
		\end{equation*}
		\item With the concrete choice of parameters $\tilde{C}_A$, $C_\rho$ \autoref{eq:rest_item3} reads as follows
		\begin{equation} \label{eq:diffb}
			\max \set{\norm{\vx-\vx'}^n,\norm{\vy-\vy'}^n} \leq  3^{-n} \left(2^{-k} + \norm{\vx-\vy}^n\right).
		\end{equation}
		Since we assume that $\norm{\vx-\vy}^n \geq 2^{-k}$ it follows from \autoref{eq:diffb} that
		\begin{equation*}
			\max \set{ \norm{\vx-\vx'}, \norm{\vy-\vy'}}  \leq  3^{-1} \norm{\vx-\vy}.
		\end{equation*}
		This in particular shows that
		\begin{equation*}
			\norm{\vx-\vy}-\norm{\vx-\vx'}-\norm{\vy-\vy'} \geq 0.
		\end{equation*}
		We apply Jensen's inequality, which states that for $a,b,c \geq 0$
		\begin{equation} \label{eq:jensen3}
			a^n+b^n+c^n \geq 3^{1-n}(a+b+c)^n.
		\end{equation}
		We use $a=\norm{\vx+t_\vx(\vx'-\vx)-\vy-t_\vy(\vy'-\vy)}$, $b = \norm{t_\vx(\vx'-\vx)}$ and $c = \norm{t_\vy(\vy'-\vy)}$, which then (along with the triangle inequality) gives
		\begin{equation*}
			\norm{\vx+t_\vx(\vx'-\vx)-\vy-t_\vy(\vy'-\vy)}^n + \norm{t_\vx(\vx'-\vx)}^n +\norm{t_\vy(\vy'-\vy)}^n \geq 3^{1-n}\norm{\vx-\vy}^n.
		\end{equation*}
		In other words, it follows from \autoref{eq:diffb} that
		\begin{equation*}
			\begin{aligned}
				\norm{\vx+t_\vx(\vx'-\vx)-\vy-t_\vy(\vy'-\vy)}^n & \geq 3^{1-n}\norm{\vx-\vy}^n -  t_\vx^n\norm{\vx'-\vx}^n - t_\vy^n\norm{\vy'-\vy}^n\\
				& \geq 3^{1-n}\norm{\vx-\vy}^n - \norm{\vx'-\vx}^n - \norm{\vy'-\vy}^n\\
				& \geq 3^{1-n}\norm{\vx-\vy}^n - 3^{-n} 2\left(2^{-k} + \norm{\vx-\vy}^n\right)\\
				&= 3^{-n} \norm{\vx-\vy}^n - 3^{-n} 2^{1-k}.
			\end{aligned}
		\end{equation*}
	\end{itemize}
\end{proof}

The approximation to the identity in \autoref{def:wavelet} can be used to construct \emph{wavelet frames} that can approximate arbitrary functions in $\mathcal{L}^1(\R^n)$, as shown in the theorem below.
\begin{remark}
	When $\norm{\vx - \vy}^n < 2^{-k}$, \autoref{eq:kk} also holds since the right hand side of \autoref{eq:kk} is negative, so this inequality is trivial.
\end{remark}

\begin{theorem}[\cite{ShaCloCoi18}] \label{le:WaveletApprox} Let $(S_k:\R^n \times \R^n \to \R)_{k \in \Z}$ be a symmetric AtI which
	satisfies the double Lipschitz condition (see \autoref{eq:DoubleLipschitzCondition}).
	Let
	\begin{equation} \label{eq:frame}
		\psi_{k,\vy}(\vx) := 2^{-k/2}\left(S_k(\vx,\vy)-S_{k-1}(\vx,\vy)\right)
		\text{ for all } \vx,\vy \in \R^n  \text{ and } k \in \Z.
	\end{equation}
	Then the set of functions
	\begin{equation} \label{eq:waveletframe}
		\mathcal{W} := \set{ \vx \to \psi_{k,\vec{b}}(\vx): k \in \Z , \vec{b} \in 2^{-k/n}\Z^n },
	\end{equation}
	is a frame and for every function $f \in \mathcal{L}^1(\R^n)$ there exists a linear combination of $N$ elements of $\mathcal{W}$, denoted by $f_N$, satisfying
	\begin{equation} \label{eq:app_general}
		\norm{f-f_N}_{L^2} \leq \norm{f}_{\mathcal{L}^1} (N+1)^{-1/2}.
	\end{equation}
\end{theorem}

\subsection*{Acknowledgements}
%
%

This research was funded in whole, or in part, by the Austrian Science Fund
(FWF) 10.55776/P34981 (OS \& LF) -- New Inverse Problems of Super-Resolved Microscopy (NIPSUM),
SFB 10.55776/F68 (OS) ``Tomography Across the Scales'', project F6807-N36
(Tomography with Uncertainties), and 10.55776/T1160 (CS) ``Photoacoustic Tomography: Analysis and Numerics''. For open access purposes, the author has applied a CC BY public copyright license to any author-accepted manuscript version arising from this submission.
The financial support by the Austrian Federal Ministry for Digital and Economic
Affairs, the National Foundation for Research, Technology and Development and the Christian Doppler
Research Association is gratefully acknowledged.

%

The authors would like to thank some referees for their valuable suggestions and their patience.
	
\section*{References}
\printbibliography[heading=none]
\end{document}

%% file: header.tex
\newtheorem{lemma}{Lemma}[section]

\newaliascnt{proposition}{lemma}

\aliascntresetthe{proposition}

\newaliascnt{corollary}{lemma}
\newtheorem{corollary}[corollary]{Corollary}
\aliascntresetthe{corollary}

\newaliascnt{theorem}{lemma}
\newtheorem{theorem}[theorem]{Theorem}
\aliascntresetthe{theorem}

\newaliascnt{definition}{lemma}
\newtheorem{definition}[definition]{Definition}
\aliascntresetthe{definition}

\newaliascnt{example}{lemma}
\newtheorem{example}[example]{Example}
\aliascntresetthe{example}

\newaliascnt{assumption}{lemma}
\newtheorem{assumption}[assumption]{Assumption}
\aliascntresetthe{assumption}

\newaliascnt{notation}{lemma}

\aliascntresetthe{notation}

\newtheorem{remark}{Remark}

\theoremsymbol{\ensuremath{\square}}
\newtheorem*{proof}{Proof}

\newcommand{\N}{\mathbb{N}}
\newcommand{\Z}{\mathbb{Z}}
\newcommand{\R}{\mathbb{R}}

\newcommand{\A}{F}

\newcommand{\bP}{{\bf P}}

\newcommand{\bx}{{\bf x}}
\newcommand{\by}{{\bf y}}

\newcommand{\vp}{{\vec p}}

\let\RE\Re
\let\Re=\undefined
\DeclareMathOperator{\Re}{\RE e}
\let\IM\Im
\let\Im=\undefined
\DeclareMathOperator{\Im}{\IM m}

\newcommand{\domain}[1]{\mathcal{D}\left(#1\right)}
\newcommand{\abs}[1]{\left|#1\right|}
\newcommand{\norm}[1]{\left\|#1\right\|}
\newcommand{\set}[1]{\left\{ #1\right\}}

\newcommand{\e}{\mathrm e}

\newcommand{\vx}{{\vec{x}}}
\newcommand{\vy}{{\vec{y}}}
\newcommand{\vw}{\vec{w}}
\newcommand{\vz}{\vec{z}}

\newcommand{\vfm}{{\bf{f}}}
\newcommand{\vzm}{{\bf{z}}}

\newcommand{\bw}{{\bf w}}

\newcommand{\interval}{\texttt{I}}
\newcommand{\quader}{\texttt{Q}}

\newcommand{\scinn}[1]{S^C_{#1}}
\newcommand{\wcinn}[1]{\psi^C_{#1}}

\newcommand\blfootnote[1]{%
	\begingroup
	\renewcommand\thefootnote{}\footnote{#1}%
	\addtocounter{footnote}{-1}%
	\endgroup
}

